\documentclass[12pt]{amsart}

\usepackage{amsmath,amssymb,amsthm}
\usepackage{mathtools}
\usepackage{microtype}
\usepackage[hidelinks]{hyperref}
\usepackage[margin=1.00in]{geometry}

\numberwithin{equation}{section}
\newtheorem{theorem}{Theorem}[section]
\newtheorem{lemma}[theorem]{Lemma}
\newtheorem{proposition}[theorem]{Proposition}
\newtheorem{corollary}[theorem]{Corollary}

\newcommand{\Irr}{\operatorname{Irr}}
\newcommand{\F}{\mathbb F}

\title{Character Nonvanishing under Coprime Actions}
\author[E. Hou]{Eric Hou}
\email{eric.x.hou@gmail.com}
\date{}

\begin{document}

\begin{abstract}
Let a finite group $A$ act coprimely on a finite group $G$, and let $C=C_G(A)$. Navarro asked whether an $A$-invariant irreducible character has linear Glauberman--Isaacs correspondent if and only if it is nonvanishing on $C$. We give counterexamples to this equivalence. We also prove that when $G$ is nilpotent, nonvanishing on $C$ implies that the Glauberman--Isaacs correspondent is linear.
\end{abstract}


\maketitle

\section{Introduction}
Let $A$ act coprimely on a finite group $G$, and let
\[
 C=C_G(A).
\]
Write
\[
 \Irr_A(G)=\{\chi\in\Irr(G)\mid \chi^a=\chi\text{ for all }a\in A\}.
\]
The Glauberman--Isaacs correspondence gives a canonical bijection
\[
 \Irr_A(G)\longrightarrow \Irr(C),\qquad \chi\longmapsto\chi^*;
\]
see \cite{Glauberman1968,Isaacs1973}. Navarro asked whether
\[
 \chi^*(1)=1
 \quad\Longleftrightarrow\quad
 \chi(c)\ne0\text{ for every }c\in C;
\]
see \cite[Problem~6.3]{Navarro2023}. He noted that even the case where $G$ is nilpotent was unknown and might admit counterexamples. Isaacs also discusses the same question in \cite[Section~7]{Isaacs2022} and proves the implication from linearity of the correspondent to nonvanishing when the acting group is a $q$-group.

The implication
\[
 \chi^*(1)=1
 \quad\Longrightarrow\quad
 \chi_C\text{ is nowhere zero}
\]
is false even when $G$ is a $p$-group. For every prime $p$, we construct a $p$-group $G$ and a cyclic $p'$-group
\[
 A\cong C_{qr},
\]
where $q$ and $r$ are distinct primes, such that $C_G(A)$ is abelian and some $\chi\in\Irr_A(G)$ vanishes on $C_G(A)$. Since $C_G(A)$ is abelian, every Glauberman--Isaacs correspondent is linear. Thus the failure already occurs for nilpotent $G$ and cyclic nilpotent $A$. These examples also give counterexamples to Kourovka Problem~21.100 \cite{Kourovka}.

The opposite implication is valid for nilpotent groups. We prove that if $G$ is nilpotent, then
\[
 \chi_C\text{ is nowhere zero}
 \quad\Longrightarrow\quad
 \chi^*(1)=1
\]
for every coprime action of $A$ on $G$. For $A=1$ this is Burnside's theorem that a nonlinear irreducible character has a zero. We prove the general statement by descending through vanishing subgroups of irreducible characters of nilpotent groups. Since the intermediate subgroups need not be normal in $G$, the induction tracks vanishing on all $G$-conjugates of a suitable fixed point.

If $A$ is a $q$-group, then for arbitrary $G$ the converse implication
\[
 \chi^*(1)=1
 \quad\Longrightarrow\quad
 \chi_C\text{ is nowhere zero}
\]
follows from the $q$-group form of the Glauberman correspondence; see \cite[Theorem~13.14]{Isaacs1976}. Hence, when $G$ is nilpotent and $A$ is a $q$-group, the two conditions are equivalent. Thus the two implications are governed by different hypotheses: one by a restriction on $A$, the other by nilpotence of $G$.

The nilpotent theorem also gives a counting consequence. If
\[
 \mathcal N_A(G)=\{\chi\in\Irr_A(G)\mid \chi(c)\ne0\text{ for every }c\in C\},
\]
then for nilpotent $G$,
\[
 |\mathcal N_A(G)|\le |C:C'|.
\]
Equality holds when $A$ is a $q$-group. Our counterexamples show that strict inequality can occur even for $p$-groups, and direct powers make the ratio $|\mathcal N_A(G)|/|C:C'|$ arbitrarily small.

For odd $p$, the same examples give odd-order counterexamples to Navarro's Carter-subgroup problem \cite[Problem~6.7]{Navarro2023}. A direct-product modification gives examples with nonabelian fixed-point and Carter subgroups.

Section~2 constructs the counterexamples. Section~3 derives counting and Carter-subgroup consequences. Section~4 treats the two implications separately.

\section{Counterexamples to Navarro's criterion}
The construction is uniform over $\F_p$ and is based on an $A$-invariant partition of a vector space into $p$ equal parts.

\begin{theorem}\label{thm:uniform}
Let $p$ be a prime. There exist a finite $p$-group $G$ and a cyclic $p'$-group $A$ acting faithfully and coprimely on $G$ such that, with $C=C_G(A)$,
\[
 C\text{ is abelian}
 \qquad\text{and}\qquad
 |\mathcal N_A(G)|<|C/C'|,
\]
where
\[
 \mathcal N_A(G)=\{\chi\in\Irr_A(G)\mid \chi(c)\ne0\text{ for all }c\in C\}.
\]
In particular, Kourovka Problem~21.100 \cite{Kourovka} and Navarro's Problem~6.3 \cite{Navarro2023} have counterexamples with $G$ a $p$-group for every prime $p$.
\end{theorem}

\begin{proof}
Choose distinct odd primes $q,r\ne p$, and let
\[
 a=\operatorname{ord}_q(p),\qquad b=\operatorname{ord}_r(p).
\]
For positive integers $m,n$, set
\[
 U=\F_{p^{am}},\qquad W=\F_{p^{bn}},\qquad V=U\oplus W.
\]
Since $q\mid p^{am}-1$ and $r\mid p^{bn}-1$, choose
\[
 \omega\in U^\times,\qquad \zeta\in W^\times
\]
of orders $q$ and $r$, respectively. Let
\[
 A=\langle x\rangle\cong C_{qr}
\]
act on $V$ by
\[
 x\cdot(u,w)=(\omega u,\zeta w).
\]
This action is faithful and $C_V(A)=0$.

Let
\[
 X=p^{am},\qquad Y=p^{bn},\qquad T=\frac{|V|}{p}=\frac{XY}{p}.
\]
The $A$-orbits on $V$ have sizes $1,q,r,qr$. The numbers of orbits of sizes $q,r,qr$ are
\[
 N_q=\frac{X-1}{q},\qquad
 N_r=\frac{Y-1}{r},\qquad
 N_{qr}=\frac{(X-1)(Y-1)}{qr}.
\]

Choose integers $\alpha,\beta$ with
\[
 0\le \alpha<r,\qquad 0\le \beta<q,
\]
such that
\[
 \alpha q\equiv T\pmod r,
 \qquad
 \beta r\equiv T\pmod q.
\]
Then
\[
 \gamma=\frac{T-\alpha q-\beta r}{qr}
\]
is an integer. For sufficiently large $m,n$, we have $\gamma\ge0$ and
\[
 N_q\ge (p-1)\alpha,
 \qquad
 N_r\ge (p-1)\beta,
 \qquad
 N_{qr}\ge (p-1)\gamma.
\]
The first two inequalities follow by increasing $m$ and $n$. For the third,
\[
 qr\bigl(N_{qr}-(p-1)\gamma\bigr)
 =T-X-Y+1+(p-1)(\alpha q+\beta r),
\]
and the right-hand side tends to infinity with $m,n$.

Choose pairwise disjoint $A$-invariant subsets
\[
 \Omega_1,\ldots,\Omega_{p-1}\subseteq V,
\]
each consisting of $\alpha$ orbits of size $q$, $\beta$ orbits of size $r$, and $\gamma$ orbits of size $qr$. Thus
\[
 |\Omega_i|=\alpha q+\beta r+\gamma qr=T
 \qquad (1\le i\le p-1).
\]
Let
\[
 \Omega_0=V\setminus\bigcup_{i=1}^{p-1}\Omega_i.
\]
Since $|V|=pT$, we also have $|\Omega_0|=T$. Hence
\[
 V=\Omega_0\sqcup\Omega_1\sqcup\cdots\sqcup\Omega_{p-1}
\]
is an $A$-invariant partition into $p$ equal parts.

Let
\[
 B=\{f:V\longrightarrow\F_p\},
\]
written additively, and let $V$ act on $B$ by translations
\[
 (\tau_t f)(v)=f(v+t).
\]
Set
\[
 G=B\rtimes V.
\]
Then $G$ is a $p$-group. The action of $A$ on $V$ extends to $G$ by
\[
 x\cdot(f,t)=(f\circ x^{-1},x\cdot t).
\]
Since $C_V(A)=0$,
\[
 C=C_G(A)=B^A,
\]
and $C$ is elementary abelian.

Let $\varepsilon$ be a primitive $p$th root of unity. For $u\in B$, define
\[
 \lambda_u(f)=\varepsilon^{\sum_{v\in V}u(v)f(v)}.
\]
Take $u=\delta_0$. Its stabilizer under the translation action of $V$ is trivial. Hence the inertia group of $\lambda_u$ in $G$ is $B$, and
\[
 \theta=(\lambda_u)^G\in\Irr(G),
 \qquad
 \theta(1)=|V|.
\]
Since $u$ is fixed by $A$, the character $\theta$ is $A$-invariant.

Define $c\in B$ by
\[
 c(v)=i\qquad\text{if }v\in\Omega_i,
 \qquad 0\le i\le p-1,
\]
where $i$ is viewed in $\F_p$. The partition is $A$-invariant, so $c\in B^A=C$. By the induced-character formula,
\[
 \theta(c)=\sum_{t\in V}\varepsilon^{c(t)}
 =T\sum_{i=0}^{p-1}\varepsilon^i
 =0.
\]
Thus $\theta\in\Irr_A(G)$ vanishes on $C$.

Since $A$ is solvable and $(|A|,|G|)=1$, the Glauberman--Isaacs correspondence gives a bijection
\[
 \Irr_A(G)\longrightarrow\Irr(C);
\]
see \cite{Glauberman1968,Isaacs1973}. As $C$ is abelian,
\[
 |\Irr_A(G)|=|\Irr(C)|=|C|=|C/C'|.
\]
Since $\theta\notin\mathcal N_A(G)$,
\[
 |\mathcal N_A(G)|<|C/C'|.
\]
Moreover, the correspondent of every member of $\Irr_A(G)$ is linear, while $\theta$ vanishes on $C$. This gives the asserted counterexample to Navarro's criterion.
\end{proof}

Here are explicit choices for the two smallest primes. If $p=2$, take $q=3$, $r=7$, and
\[
 V=\F_4^2\oplus\F_8,
\]
with $A=C_{21}$ acting by scalars of orders $3$ and $7$. The orbit sizes are
\[
 1,3,3,3,3,3,7,21,21,21,21,21,
\]
and
\[
 1+21+21+21=64=|V|/2.
\]
If $p=3$, take $q=5$, $r=11$, and
\[
 V=\F_{3^4}\oplus\F_{3^5}.
\]
Then $A=C_{55}$, the orbit sizes are $1$, $5^{(16)}$, $11^{(22)}$, and $55^{(352)}$, and
\[
 |V|/3=6561=5+11+119\cdot55.
\]
Two disjoint unions of one $5$-orbit, one $11$-orbit, and $119$ $55$-orbits have size $6561$, and the complement has the same size.

\section{Counting consequences and Carter subgroups}

\begin{corollary}\label{cor:smallA}
In Theorem~\ref{thm:uniform}, the operator group $A$ may always be chosen with $|A|\le 35$.
\end{corollary}

\begin{proof}
If $p\notin\{3,5\}$, choose $(q,r)=(3,5)$. If $p=3$, choose $(q,r)=(5,7)$, and if $p=5$, choose $(q,r)=(3,7)$. The proof of Theorem~\ref{thm:uniform} then applies with $A\cong C_{qr}$.
\end{proof}

\begin{corollary}\label{cor:carter}
For every odd prime $p$, there exists a finite solvable group $\Gamma$ of odd order with an abelian Carter subgroup $K$ such that
\[
 \bigl|\{\chi\in\Irr(\Gamma)\mid
 \chi(k)\ne0\text{ for all }k\in K\}\bigr|
 <|K/K'|.
\]
Consequently, Navarro's Problem~6.7 \cite{Navarro2023} has a negative answer for groups of odd order, and the head characters of $\Gamma$ are not exactly the Carter-nonvanishing irreducible characters.
\end{corollary}

\begin{proof}
Use Theorem~\ref{thm:uniform} with $p$ odd. The auxiliary primes $q$ and $r$ are odd, so both $G$ and $A=C_{qr}$ have odd order. Let
\[
 \Gamma=G\rtimes A,
 \qquad
 K=C_G(A)\times A.
\]
The subgroup $K$ is abelian. Since $A$ is the unique Hall $p'$-subgroup of $K$, it is characteristic in $K$. If $g\in G$ normalizes $A$, then for every $a\in A$ the commutator $gag^{-1}a^{-1}$ lies in both $G$ and $A$, and hence is trivial. Thus $g\in C_G(A)$, so
\[
 N_\Gamma(A)=C_G(A)A=K.
\]
It follows that $N_\Gamma(K)=K$, and hence $K$ is a Carter subgroup of $\Gamma$.

Let $\chi\in\Irr(\Gamma)$ be nonvanishing on $K$, and let $\varphi\in\Irr(G)$ lie under $\chi$. If $\varphi$ is not $A$-invariant, its inertia group is $GA_\varphi<G A=\Gamma$. By Clifford correspondence, $\chi$ is induced from an irreducible character of $GA_\varphi$. If $a\in A\setminus A_\varphi$, then no conjugate of $a$ lies in $GA_\varphi$, since its image in $\Gamma/G\cong A$ is still $a$. Hence $\chi(a)=0$, contrary to the choice of $\chi$. Therefore $\varphi\in\Irr_A(G)$.

Since $\Gamma/G\cong A$ is cyclic, $\varphi$ extends to $\Gamma$ by \cite[Corollary~11.22]{Isaacs1976}. Gallagher's theorem \cite[Corollary~6.17]{Isaacs1976} shows that the irreducible characters of $\Gamma$ above $\varphi$ are the $|A|$ twists of an extension by the linear characters of $A$. If one of these characters is nonvanishing on $K$, then $\varphi$ is nonvanishing on $C_G(A)$. Hence
\[
 \bigl|\{\chi\in\Irr(\Gamma)\mid
 \chi(k)\ne0\text{ for all }k\in K\}\bigr|
 \le |A|\,|\mathcal N_A(G)|.
\]
By Theorem~\ref{thm:uniform},
\[
 |A|\,|\mathcal N_A(G)|<|A|\,|C_G(A)|=|K|=|K/K'|.
\]
The last assertion follows from \cite[Theorem~A]{Isaacs2022}, which gives a bijection between the head characters and the linear characters of a Carter subgroup.
\end{proof}

\begin{corollary}\label{cor:ratio}
For every prime $p$ and every $\varepsilon>0$, there is a coprime action of a cyclic $p'$-group $A$ on a finite $p$-group $H$ such that, with $D=C_H(A)$,
\[
 \frac{|\mathcal N_A(H)|}{|D/D'|}<\varepsilon.
\]
If $p$ is odd, there is also a finite solvable group $\Lambda$ of odd order with a Carter subgroup $L$ such that
\[
 \frac{|\{\chi\in\Irr(\Lambda)\mid \chi(\ell)\ne0\text{ for all }\ell\in L\}|}{|L/L'|}<\varepsilon.
\]
\end{corollary}

\begin{proof}
Take $G,A,C$ from Theorem~\ref{thm:uniform} and let
\[
 \rho=\frac{|\mathcal N_A(G)|}{|C|}<1.
\]
Let $A$ act diagonally on $H=G^n$. Then $C_H(A)=C^n$. Every irreducible character of $H$ is a tensor product of irreducible characters of the factors, and such a character is $A$-invariant and nonvanishing on $C^n$ exactly when every factor lies in $\mathcal N_A(G)$. Hence
\[
 |\mathcal N_A(H)|=|\mathcal N_A(G)|^n
\]
and
\[
 \frac{|\mathcal N_A(H)|}{|C_H(A)/C_H(A)'|}=\rho^n.
\]
Choose $n$ so that $\rho^n<\varepsilon$.

If $p$ is odd, let $\Lambda=H\rtimes A$ and $L=C^n\times A$. The proof of Corollary~\ref{cor:carter} shows that $L$ is an abelian Carter subgroup and that the number of irreducible characters of $\Lambda$ nonvanishing on $L$ is at most $|A|\,|\mathcal N_A(H)|$. The corresponding ratio is therefore at most $\rho^n<\varepsilon$.
\end{proof}

\begin{corollary}\label{cor:nonabelian}
For every prime $p$, there are counterexamples to Problem~21.100 with $G$ a $p$-group and $C_G(A)$ nonabelian. For every odd prime $p$, there are odd-order counterexamples to the Carter-subgroup question with a nonabelian Carter subgroup.
\end{corollary}

\begin{proof}
Let $G,A,C$ be as in Theorem~\ref{thm:uniform}, and let $Q$ be any nonabelian $p$-group on which $A$ acts trivially. Then
\[
 C_{G\times Q}(A)=C\times Q.
\]
An irreducible character of $G\times Q$ is $A$-invariant and nonvanishing on $C\times Q$ exactly when its $G$-factor lies in $\mathcal N_A(G)$ and its $Q$-factor is nonvanishing on $Q$. By Burnside's theorem on zeros of nonlinear irreducible characters, the latter characters are precisely the linear characters of $Q$. Hence
\[
 |\mathcal N_A(G\times Q)|=|\mathcal N_A(G)|\,|Q/Q'|
 <|C|\,|Q/Q'|
 =|(C\times Q)/(C\times Q)'|.
\]
For odd $p$, pass to
\[
 \Lambda=(G\times Q)\rtimes A,
 \qquad
 L=(C\times Q)\times A.
\]
The subgroup $L$ is nilpotent and self-normalizing by the same normalizer argument as in Corollary~\ref{cor:carter}, and hence is a nonabelian Carter subgroup. The Clifford--Gallagher argument used there also gives
\[
 \bigl|\{\chi\in\Irr(\Lambda)\mid
 \chi(\ell)\ne0\text{ for all }\ell\in L\}\bigr|
 \le |A|\,|\mathcal N_A(G\times Q)|.
\]
Therefore
\[
 \bigl|\{\chi\in\Irr(\Lambda)\mid
 \chi(\ell)\ne0\text{ for all }\ell\in L\}\bigr|
 < |A|\,|(C\times Q)/(C\times Q)'|
 =|L/L'|,
\]
which gives the required Carter-subgroup counterexample.
\end{proof}

\section{Nonvanishing and linear correspondents}

\subsection{The \texorpdfstring{$q$}{q}-group case}
For $q$-group actions, the following implication is known and requires no nilpotence assumption on $G$.

\begin{proposition}\label{prop:qgroup-known}
Let $A$ be a $q$-group acting coprimely on a finite group $G$, let $C=C_G(A)$, and let $\chi\in\Irr_A(G)$. If the Glauberman correspondent $\chi^*\in\Irr(C)$ is linear, then
\[
 \chi(c)\ne0
 \qquad\text{for every }c\in C.
\]
\end{proposition}

\begin{proof}
By the $q$-group Glauberman correspondence, see \cite[Theorem~13.14(a) and Corollary~13.21]{Isaacs1976} and \cite[Section~7]{Isaacs2022}, there are an integer $e\ge1$ with $e\equiv\pm1\pmod q$ and a character $\Delta$ of $C$, possibly zero, such that
\[
 \chi_C=e\chi^*+q\Delta.
\]
Suppose that $\chi(c)=0$ for some $c\in C$. Choose a number field containing all relevant character values, let $R$ be its ring of integers, and let $M$ be a maximal ideal of $R$ lying above $q$. Then
\[
 e\chi^*(c)=-q\Delta(c)\in M.
\]
Since $M\cap\mathbb Z=q\mathbb Z$, we have $e\notin M$. Also, $\chi^*(c)$ is a root of unity because $\chi^*$ is linear, and hence is a unit of $R$. Thus $e\chi^*(c)\notin M$, a contradiction.
\end{proof}

\subsection{Nilpotent groups}
For the rest of this section, let $G$ be nilpotent, let $A$ act coprimely on $G$, and set $C=C_G(A)$. If $H$ is an $A$-invariant subgroup with $C\le H$, then $C_H(A)=C$. For $\psi\in\Irr_A(H)$, we write $\psi^*$ for the Glauberman--Isaacs correspondent for the action of $A$ on $H$.

For $\psi\in\Irr(H)$, define
\[
 V(\psi)=\langle x\in H\mid \psi(x)\ne0\rangle.
\]
Thus $V(\psi)$ is a normal subgroup of $H$, and $\psi$ vanishes on $H\setminus V(\psi)$.

\begin{lemma}\label{lem:linear-correspondent}
Let $H$ be an $A$-invariant subgroup of $G$ with $C\le H$, and let $\psi\in\Irr_A(H)$. If $\psi$ is linear, then $\psi^*$ is linear.
\end{lemma}

\begin{proof}
Since $\psi$ is linear, its restriction to the trivial subgroup is the irreducible character $1_1$. By \cite[Lemma~7.2(b)]{Isaacs2022} (see also \cite[Theorem~A]{IsaacsNavarro1991}), applied with $N=1$,
\[
 (\psi^*)_1=(1_1)^*=1_1.
\]
Hence $\psi^*(1)=1$.
\end{proof}

\begin{lemma}\label{lem:vanishing-subgroup}
Let $H$ be a finite nilpotent group and let $\psi\in\Irr(H)$ be nonlinear. Then
\[
 V(\psi)H'<H.
\]
\end{lemma}

\begin{proof}
By \cite[Corollary~6.14]{Isaacs1976}, every nilpotent group is an $M$-group. Hence
\[
 \psi=\lambda^H
\]
for a linear character $\lambda$ of a proper subgroup $K<H$. Choose a maximal subgroup $M$ of $H$ containing $K$. Since $H$ is nilpotent, $M\lhd H$ and $|H:M|$ is prime. The induced character $\lambda^H$ vanishes outside the union of the conjugates of $K$, all of which lie in $M$. Thus $V(\psi)\le M$. Also $H'\le M$, since $H/M$ is cyclic. Therefore $V(\psi)H'\le M<H$.
\end{proof}

Nilpotence of the ambient group enters through the following lemma.

\begin{lemma}\label{lem:avoidance}
Let $H$ be an $A$-invariant subgroup of $G$ with $C\le H$, and let $L\lhd H$ be $A$-invariant. Assume that
\[
 L<H
 \qquad\text{and}\qquad
 CL=H.
\]
Then there exists $c\in C$ such that no $G$-conjugate of $c$ lies in $L$.
\end{lemma}

\begin{proof}
We induct on $|G|$. Let $Z$ be a minimal nontrivial $A$-invariant subgroup of $Z(G)$.

Suppose first that $LZ<HZ$. Since $G$ is solvable, Glauberman's lemma gives
\[
 C_{G/Z}(A)=CZ/Z.
\]
In $G/Z$, the subgroup $LZ/Z$ is $A$-invariant and normal in $HZ/Z$, it is proper in $HZ/Z$, and
\[
 (CZ/Z)(LZ/Z)=HZ/Z.
\]
By induction, there is a coset $cZ\in CZ/Z$, with $c\in C$, having no $G/Z$-conjugate in $LZ/Z$. Hence no $G$-conjugate of $c$ lies in $L$.

Now suppose that $LZ=HZ$. Since $L<H$, we have $Z\nleq L$. Dedekind's identity gives
\[
 H=L(H\cap Z).
\]
Thus $H\cap Z\nleq L$. By the minimality of $Z$, we have $Z\le H$, and then $Z\cap L=1$. Since $CL=H$, every coset of $L$ in $H$ has a representative in $C$. Hence $A$ fixes every coset of $L$ in $H$. For $z\in Z$ and $a\in A$,
\[
 z^{-1}z^a\in Z\cap L=1,
\]
so $Z\le C$. Any nonidentity $c\in Z$ is central, belongs to $C$, and does not lie in $L$. Therefore no $G$-conjugate of $c$ lies in $L$.
\end{proof}

Because the intermediate subgroup need not be normal in $G$, we use the following stronger induction statement.

\begin{proposition}\label{prop:descent}
Let $H$ be an $A$-invariant subgroup of $G$ with $C\le H$, and let $\psi\in\Irr_A(H)$. If $\psi^*$ is nonlinear, then there exists $c\in C$ such that
\[
 \psi(x)=0
\]
for every $x\in H$ that is $G$-conjugate to $c$.
\end{proposition}

\begin{proof}
We induct on $|H|$. By Lemma~\ref{lem:linear-correspondent}, $\psi$ is nonlinear. Let
\[
 R=V(\psi)H'.
\]
Then $R$ is normal in $H$, is $A$-invariant, and satisfies $R<H$ by Lemma~\ref{lem:vanishing-subgroup}.

Suppose first that $C\nleq R$. Let
\[
 Y=H/R.
\]
Since $H'\le R$, the group $Y$ is abelian. The standard coprime decomposition gives
\[
 Y=C_Y(A)\times [Y,A].
\]
Since $H$ is solvable, Glauberman's lemma gives
\[
 C_Y(A)=CR/R.
\]
Let $L$ be the inverse image of $[Y,A]$ in $H$. Then $L$ is $A$-invariant and normal in $H$, and
\[
 L<H,\qquad CL=H,\qquad V(\psi)\le R\le L.
\]
By Lemma~\ref{lem:avoidance}, there exists $c\in C$ such that no $G$-conjugate of $c$ lies in $L$. If $x\in H$ is $G$-conjugate to $c$, then $x\notin V(\psi)$, and therefore $\psi(x)=0$.

Suppose now that $C\le R$. The irreducible constituents of $\psi_R$ form one $H$-orbit. By Glauberman's lemma \cite[Lemma~13.8 and Corollary~13.9]{Isaacs1976}, the $A$-invariant constituents form one $C$-orbit. Since $C\le R$, conjugation by $C$ acts trivially on $\Irr(R)$, so $\psi_R$ has a unique $A$-invariant constituent; call it $\xi$. By \cite[Lemma~7.2(a)]{Isaacs2022}, $\psi^*$ lies over $\xi^*$. But
\[
 C_R(A)=C_H(A)=C,
\]
so $\psi^*$ and $\xi^*$ are irreducible characters of the same group $C$. Hence
\[
 \xi^*=\psi^*,
\]
and $\xi^*$ is nonlinear.

Apply induction to $(R,\xi)$. There exists $c\in C$ such that
\[
 \xi(y)=0
\]
for every $y\in R$ that is $G$-conjugate to $c$. Let $x\in H$ be $G$-conjugate to $c$. If $x\notin R$, then $x\notin V(\psi)$ and $\psi(x)=0$. If $x\in R$, Clifford theory gives
\[
 \psi_R=e\sum_j \xi^{h_j}
\]
for a positive integer $e$ and representatives $h_j\in H$ of the $H$-orbit of $\xi$. Thus
\[
 \psi(x)=e\sum_j \xi(h_jxh_j^{-1}).
\]
Each $h_jxh_j^{-1}$ lies in $R$ and is $G$-conjugate to $c$, so every term in the sum is zero. Hence $\psi(x)=0$.
\end{proof}

\begin{theorem}\label{thm:nilpotent}
Let $A$ act coprimely on a finite nilpotent group $G$, let $C=C_G(A)$, and let $\chi\in\Irr_A(G)$. If
\[
 \chi(c)\ne0
 \qquad\text{for every }c\in C,
\]
then the Glauberman--Isaacs correspondent $\chi^*\in\Irr(C)$ is linear.
\end{theorem}

\begin{proof}
If $\chi^*$ were nonlinear, Proposition~\ref{prop:descent} with $H=G$ would give $c\in C$ such that $\chi(c)=0$, contrary to the hypothesis.
\end{proof}

\begin{corollary}\label{cor:nilpotent-count}
Let $A$ act coprimely on a finite nilpotent group $G$, and let $C=C_G(A)$. Then
\[
 |\mathcal N_A(G)|\le |C:C'|.
\]
If $A$ is a $q$-group, then
\[
 |\mathcal N_A(G)|=|C:C'|,
\]
and for every $\chi\in\Irr_A(G)$,
\[
 \chi^*(1)=1
 \quad\Longleftrightarrow\quad
 \chi(c)\ne0\text{ for every }c\in C.
\]
\end{corollary}

\begin{proof}
The Glauberman--Isaacs correspondence is a bijection from $\Irr_A(G)$ onto $\Irr(C)$. By Theorem~\ref{thm:nilpotent}, every member of $\mathcal N_A(G)$ corresponds to a linear character of $C$. Hence
\[
 |\mathcal N_A(G)|\le |\operatorname{Lin}(C)|=|C:C'|.
\]
If $A$ is a $q$-group, Proposition~\ref{prop:qgroup-known} gives the converse: every $A$-invariant irreducible character whose correspondent is linear is nonvanishing on $C$. The last two assertions follow.
\end{proof}

Proposition~\ref{prop:qgroup-known} assumes that $A$ is a $q$-group and allows arbitrary $G$, whereas Theorem~\ref{thm:nilpotent} assumes that $G$ is nilpotent and allows arbitrary coprime $A$. The $q$-group hypothesis in Proposition~\ref{prop:qgroup-known} cannot be replaced uniformly by nilpotence of $A$: Theorem~\ref{thm:uniform} has $G$ a $p$-group and $A\cong C_{qr}$ cyclic.

\section*{Acknowledgements}
The author thanks Abhinav Namboori for helpful discussions and comments on the exposition.

OpenAI's ChatGPT was used for mathematical exploration and editorial assistance. The author checked the final arguments and are responsible for the contents of the paper.




\begin{thebibliography}{99}

\bibitem{Glauberman1968}
G.~Glauberman,
Correspondences of characters for relatively prime operator groups,
\textit{Canad. J. Math.} \textbf{20} (1968), 1465--1488.

\bibitem{Isaacs1973}
I.~M.~Isaacs,
Characters of solvable and symplectic groups,
\textit{Amer. J. Math.} \textbf{95} (1973), 594--635.

\bibitem{Isaacs1976}
I.~M.~Isaacs,
\emph{Character Theory of Finite Groups},
Academic Press, New York, 1976.

\bibitem{IsaacsNavarro1991}
I.~M.~Isaacs and G.~Navarro,
Character correspondences and irreducible induction and restriction,
\textit{J. Algebra} \textbf{140} (1991), 131--140.

\bibitem{Isaacs2022}
I.~M.~Isaacs,
Carter subgroups, characters and composition series,
\textit{Trans. Amer. Math. Soc. Ser. B} \textbf{9} (2022), 470--498.

\bibitem{Kourovka}
E.~I.~Khukhro and V.~D.~Mazurov (eds.),
\emph{Unsolved Problems in Group Theory: The Kourovka Notebook},
No.~21, Sobolev Institute of Mathematics, Novosibirsk, 2026,
Problem~21.100, p.~178.

\bibitem{Navarro2023}
G.~Navarro,
Problems on characters: solvable groups,
\textit{Publ. Mat.} \textbf{67} (2023), 173--198.

\end{thebibliography}
\end{document}